\documentclass[twoside,a4paper,10pt]{amsart}
\usepackage{amsmath}
\usepackage{amsthm}
\usepackage[applemac]{inputenc} 
\usepackage[english]{babel}
\usepackage[all]{xy}
\usepackage{amssymb}
\usepackage{graphicx}
\usepackage{enumitem}
\usepackage{amsfonts}
\usepackage{array}
\usepackage{xcolor}
\usepackage{graphicx}
\usepackage[textsize=tiny]{todonotes}

\usepackage{listings}
\usepackage{hyperref}

\definecolor{codegreen}{rgb}{0,0.6,0}
\definecolor{codegray}{rgb}{0.5,0.5,0.5}
\definecolor{codepurple}{rgb}{0.58,0,0.82}
\definecolor{backcolour}{rgb}{0.95,0.95,0.92}

\lstdefinestyle{mystyle}{
    backgroundcolor=\color{backcolour},   
    commentstyle=\color{codegreen},
    keywordstyle=\color{magenta},
    numberstyle=\tiny\color{codegray},
    stringstyle=\color{codepurple},
    basicstyle=\ttfamily\footnotesize,
    breakatwhitespace=false,         
    breaklines=true,                 
    captionpos=b,                    
    keepspaces=true,                 
    numbers=left,                    
    numbersep=5pt,                  
    showspaces=false,                
    showstringspaces=false,
    showtabs=false,                  
    tabsize=2
}
\newcommand{\Z}{\mathbb Z}
\renewcommand{\P}{\mathbb P}
\newcommand{\PP}{\mathbb P}
\newcommand{\Q}{\mathbb Q}

\newcommand{\cD}{\mathcal D}

\newcommand\lra{\longrightarrow}
\newcommand\edps{md-pairs\ }
\newcommand\edp{md-pair\ }
\newcommand{\lef}[1]{\lfloor #1 \rfloor}
\newcommand{\rig}[1]{\lceil #1 \rceil}

\DeclareMathOperator{\ed}{e.\!d.}

\DeclareMathOperator{\codim}{codim}
\DeclareMathOperator{\HH}{H\hspace{0.5pt}}
\DeclareMathOperator{\AAA}{A}
\DeclareMathOperator{\DA}{A}
\DeclareMathOperator{\DB}{B}
\DeclareMathOperator{\DC}{C}
\DeclareMathOperator{\DD}{D}
\DeclareMathOperator{\DE}{E}
\DeclareMathOperator{\DF}{F}
\DeclareMathOperator{\DG}{G}

\DeclareMathOperator{\cox}{\hspace{0cm}h}
\newcolumntype{C}{>{$}c<{$}}

\newtheorem{theorem}{Theorem}[section]

\newtheorem{proposition}[theorem]{Proposition}
\newtheorem{corollary}[theorem]{Corollary}
\newtheorem{lemma}[theorem]{Lemma}

\theoremstyle{definition}
\newtheorem{remark}[theorem]{Remark}
\newtheorem{definition}[theorem]{Definition}
\newtheorem{notation}[theorem]{Notation}
\newtheorem{example}[theorem]{Example}

\begin{document}
\title[Maximal disjoint Schubert cycles]
{Maximal disjoint Schubert cycles in rational homogeneous varieties}

\author[Mu\~noz]{Roberto Mu\~noz}
\address{Departamento de Matem\'atica Aplicada a las TIC, ETSISI, Universidad
Polit\'ecnica de Madrid, 28031-Madrid, Spain}
\email{roberto.munoz@upm.es}
\author[Occhetta]{Gianluca Occhetta}
\address{Dipartimento di Matematica, Universit\`a di Trento, via
Sommarive 14 I-38123 Povo di Trento (TN), Italy} \thanks{Second and third author supported by the Department of Mathematics of the University of Trento.}
\email{gianluca.occhetta@unitn.it}
\author[Sol\'a Conde]{Luis E. Sol\'a Conde}
\address{Dipartimento di Matematica, Universit\`a di Trento, via
Sommarive 14 I-38123 Povo di Trento (TN), Italy}
\email{eduardo.solaconde@unitn.it}


\begin{abstract} 
In this paper we study properties of the Chow ring of rational homogeneous varieties of classical type, more concretely, effective zero divisors of low codimension, and a related invariant called effective good divisibility. This information is then used to study the question of (non)existence of nonconstant maps among these varieties, generalizing previous results for projective spaces and Grassmannians.
\end{abstract} 
\maketitle




\maketitle

\section{Introduction}

Throughout this paper, all the varieties will be projective and defined over the field of complex numbers. Given such a  smooth variety $M$ we will denote by
$$
\AAA^{\bullet}(M):=\bigoplus_{i=0}^{\dim(M)}\AAA^{i}(M)
$$
the Chow ring of $M$; the {\it effective good divisibility of} $M$ (denoted by $\ed(M)$) measures the minimum total codimension of effective zero divisors in this ring.
More concretely, saying that $\ed(M)=e$ is to say that we cannot find two nonzero effective cycles $x_i \in \AAA^{i}(M),$  $x_j \in \AAA^{j}(M)$ with zero intersection if $i+j\leq e$, whereas two such cycles exist when the sum of their codimensions is $i+j=e+1$ (see Definitions \ref{def:edupto}, \ref{def:ed}). This invariant was introduced in \cite{MOS6} as a refined version of the good divisibility introduced by Pan in \cite{Pan}.

As we will see, the effective good divisibility of a variety $M$ may help us to understand geometric properties of $M$, such as the (non)existence of nonconstant morphisms $\varphi: M \to M'$ from $M$ to certain rational homogeneous varieties $M'$ (see Section \ref{sec:applications}). A well-known elementary result in this direction is the trivial fact that morphisms $\varphi:\P^m\to \P^n$ with $m>n$ are constant. This can be thought of as a particular case of a result due to Tango (cf \cite{Tango1}) that there are no nonconstant maps from $\P^m$ to Grassmannians of linear subspaces of $\P^n$ if $m>n$. Recently, with the concept of effective good divisibility at hand, this statement has been extended to maps between Grassmannians in \cite{NO22}.

 In this paper we will deal, more generally, with the effective good divisibility of rational homogeneous varieties of classical type. In this case, the Chow ring is well-known, and the study of the $\ed$ relies on determining disjoint Schubert cycles of maximal dimension; our study shows that $\ed$ for rational homogeneous varieties is  related to the Coxeter number of the Lie algebra of the corresponding group of automorphisms. 
Let us recall that, given a simple Lie algebra with Dynkin diagram $\cD$ of classical type, the corresponding Coxeter number (see Definition \ref{def:coxnum}) can be read out of Table \ref{tab:cox}.


\begin{table}[h!]
\caption{Coxeter numbers of Dynkin diagram of classical type.}
\centering
\begin{tabular}{|C||C|C|C|C|}
\hline
\cD &
\DA_n & 
\DB_n & 
\DC_n & 
\DD_n \\
\hline
\cox(\cD)&
n+1&
2n&
2n&
2n-2\\
\hline
\end{tabular}
\label{tab:cox}
\end{table}

We may state now the main results of this paper. First, we study the effective good divisibility of complete flag manifolds of classical type.

\begin{theorem}\label{thm:1} The effective good divisibility of a complete flag manifold of classical type $\cD$ is $\cox(\cD)-1$. 
\end{theorem}

A priori this may suggest that the effective good divisibility on any complete flag manifold equals the Coxeter number minus one. However, a computer computation with {\tt SageMath}  (see Appendix \ref{ssec:div}) shows that this is the case for type $\DG_2$, while it is not for type $\DF_4$ and $\DE_6$, for which the effective good divisibility is $12$, equal to the Coxeter number.  
The cases $\DE_7$ and $\DE_8$ are computationally out of reach with our method.

Furthermore, we study effective good divisibility of any rational homogeneous variety of classical type. It is not difficult to show (see Corollary \ref{cor:pull}) that this number is lower bounded by the effective good divisibility of the corresponding complete flag manifold. The following statement shows that the two invariants coincide in most of the cases, and identify the exceptions.
Recall (see Section \ref{ss:preliminaries} for notation and details) that any choice of a subset $R$ of the set of nodes of $\cD$ defines a parabolic subgroup and, consequently, a rational homogeneous variety, that we denote by $\cD(R)$. 

\begin{theorem}\label{thm:2} The effective good divisibility of any rational homogeneous variety of classical type is equal to the effective good divisibility of the corresponding complete flag manifold, with the following exception: in the case in which $\cD =\DD_n$ and $R\cap\{1,n-1,n\}=\emptyset$, we have $\ed(\DD_n(R))=\cox(\DD_n)=2n-2$.
\end{theorem}

We then study the question of existence of nonconstant 
maps to rational homogeneous varieties. The main idea  is that, using the homogeneity of the target, two effective zero divisors on $M'$ can be pulled-back to $M$ providing 
zero divisors in the same codimensions; this imposes conditions relating $\ed(M)$ and $\ed(M')$ for $\varphi$ not to be constant: 

\begin{theorem}\label{teo:maps2b}
Let $\cD$ be a Dynkin diagram of classical type with set of nodes $\Delta$, $R\subset \Delta$ a nonempty subset. Let $M$ be a smooth complex projective variety such that $\ed(M)>\ed(\cD(R))$. Then there are no nonconstant morphisms from $M$ to $\cD(R)$.
\end{theorem}

As a consequence we obtain the following statement: 

\begin{corollary}\label{cor:maps}
Let $\cD'$ be a Dynkin diagram of classical type, whose set of nodes is $\Delta'$ and $\cD \subsetneq \cD'$ be a proper subdiagram, with set of nodes $\Delta$. Then, for any $R \subset \Delta, R' \subset\Delta'$, any morphism $\phi:\cD'(R') \to \cD(R)$ is constant.
\end{corollary}

Corollary \ref{cor:maps}  follows from Theorem \ref{teo:maps2b}; in fact one can observe --using Theorems \ref{thm:1}, \ref{thm:2}--
that in every possible case $\ed(\cD'(R'))>\ed(\cD(R))$. 

%

This result extends and encompasses the previously quoted results, \cite{NO22, Tango1}. In fact, the nonexistence of nonconstant morphisms from a Grassmannian of linear subspaces of $\P^m$ to a Grassmannian of linear subspaces of $\P^{n}$ if $m>n$ follows now by considering the case in which $\cD'=\DA_m$, and $R'$ is a singleton. Our result may be applied to a number of different situations; for example, in the case in which $\cD'$ is of type $\DB,\DC,\DD$, it provides a similar statement also for orthogonal and isotropic Grassmannians.

Finally we note that our arguments make use of some computer computations of effective good divisibility and maximal disjoint pairs for some rational homogeneous varieties with low rank automorphism groups. We have included the corresponding scripts (that we have done with the computer software {\tt SageMath}) in the Appendix for the reader's convenience. 

After this paper was posted on ArXiv, we were informed by H. Hu, C. Li and Z. Liu that they obtained similar results,  using methods of quantum cohomology. Studying Hasse diagrams they were able to compute the effective good divisibility also in the exceptional cases. These results later appeared in the preprint \cite{HLL}. 

\section{Preliminaries}

\subsection{Effective good divisibility}\label{ssec:EGD}

Let $M$ be a smooth complex projective variety, with Chow ring $\AAA^{\bullet}(M)$.
The following definition (introduced in \cite[Definition 4.2]{MOS6} as an effective version of Pan's good divisibility, cf. \cite[Definition 4.1]{Pan}) considers the existence of effective zero divisors in this ring.

\begin{definition}
\label{def:edupto}

A variety $M$ has {\it effective good divisibility up to degree $s$} if, given effective cycles $x_i \in \AAA^{i}(M),$  $x_j \in \AAA^{j}(M)$ with $i+j = s$ and $x_i  x_j=0$, we have $x_i = 0$ or $x_j = 0$.
\end{definition}

The reason why this condition is called divisibility {\it up to} degree $s$ is the following. Consider two effective cycles $x_i \in \AAA^{i}(M),$  $x_j \in \AAA^{j}(M)$ with $i+j \leq s$ and $x_ix_j=0$, and the class $H$ of a very ample divisor. From the equality $x_ix_jH^{s-i-j}=0$, the divisibility up to degree $s$ implies that $x_i=0$, or $x_jH^{s-i-j}=0$, that is, $x_j=0$.

In particular, if $M$ has effective good divisibility up to degree $s$, then it has it up to degree $r$ for every $r\leq s$ and we can pose the following definition.

\begin{definition}
	\label{def:ed}
	The {\it effective good divisibility} of $M$, denoted by $\ed(M)$, is the maximum integer $s$ such that $M$  has  effective good divisibility up to degree $s$.
\end{definition}

\begin{definition}\label{def:Kleiman}
A pair $(x_i,x_j)$ with $x_i \in \AAA^{i}(M),$  $x_j \in \AAA^{j}(M)$ nonzero effective cycles such that $x_ix_j=0$, $i+j=\ed(M) +1$  will be called a {\it  maximal disjoint pair} (\edp for short) in $M$.
\end{definition}

\subsection{Notation: Weyl groups and homogeneous varieties}\label{ss:preliminaries}
Throughout the paper $G$ will denote a semisimple algebraic group, $B$ a {\it Borel subgroup}  of $G$, and $H$ a {\it Cartan subgroup} of $B$ (i.e., a maximal torus contained in $B$). The torus $H$ determines a {\it root system} $\Phi$, whose {\it Weyl group} $W$ is isomorphic to the quotient ${\rm N}(H)/H$ of the normalizer ${\rm N}(H)$ of $H$ in $G$. Within $\Phi$, $B$ determines a base of positive simple roots $\Delta=\{\alpha_i,\,\, i=1,\dots,n\}$, whose associated reflections we denote by $s_i\in W$.  
We denote by $\cD$ the {\it Dynkin diagram} of $G$, whose set of nodes is $\Delta$. We will assume that $\cD$ is connected, which is equivalent to say that the Lie algebra of $G$ is simple. The nodes of the Dynkin diagram $\cD$ will be numbered as in the standard reference \cite[p.~58]{Hum2} and we will identify each node $\alpha_i\in \Delta$ with the corresponding index $i$. Given any nonempty set of indices $I\subset\Delta$, we will denote by $\cD_I$ the (possibly disconnected) Dynkin subdiagram of $\cD$ whose set of nodes is $I$; the numbering of the nodes in a subdiagram $\cD_I$ of $\cD$ will be the same as in $\cD$. 

For every subset $I\subset \Delta$  we may consider the {\it parabolic subgroup} $P_I$ defined as $P_I:=BW_IB$, where $W_I\subset W$ is the subgroup of $W$ generated by the reflections $s_i$, $i\in I$. Taking the quotient by the subgroup $P_I$ we obtain the {\it rational homogeneous variety} $G/P_I$ (cf. \cite[\S~23.3]{FuHa}), which we will denote by $\cD(\Delta \setminus I)$, and represent by the Dynkin diagram $\cD$ marked in the nodes $\Delta \setminus I$. An inclusion $I\subset J\subset \Delta$ provides a smooth contraction $G/P_I\to G/P_J$.
Note that the complete flag manifold associated with $G$ can be written as
$$G/B=G/P_\emptyset=\cD(\Delta),$$
and that it dominates any $\cD(\Delta\setminus I)$ via a smooth contraction.

In the cases of classical type ($\DA_n,\DB_n,\DC_n,\DD_n$), rational homogeneous varieties have well known projective descriptions (see Table \ref {tab:RH} for some examples that we will consider in this paper): rational homogeneous varieties of types $\cD=\DA_n,\DB_n,\DC_n,\DD_n$ can be described as varieties of flags of linear spaces in $\cD(1)$.  
\begin{table}[h!]
\centering
\caption{Some rational homogeneous varieties of classical type.\label{tab:RH}} 
\begin{tabular}{|r||l|}
\hline
$\DA_n(1)$&$\PP^n$\\\hline
$\DB_n(1)$&Smooth quadric $\Q^{2n-1}$\\\hline
$\DC_n(1)$&$\PP^{2n-1}$\\\hline
$\DD_n(1)$&Smooth quadric $\Q^{2n-2}$\\\hline
$\DD_n(n-1),\DD_n(n)$&Spinor varieties (parametrizing $\PP^{n-1}$'s in $\Q^{2n-2}$)\\\hline
\end{tabular}
\end{table}

It is well known that the reflections $s_i$ with respect to the positive simple roots $\alpha_i$, $i=1,\dots,n$, generate the Weyl group $W$. Given an element $w\in W$, the smallest number $k$ for which there exist a sequence of indices $i_1,\dots,i_k\in\{1,\dots,n\}$ such that 
\[w=s_{i_1}\dots s_{i_k}\]
is called the {\it length of }$w$, and denoted by $\ell(w)$. An expression of $w$ as a product of $\ell(w)$ simple reflections is called {\it reduced}. 
Let us recall that reduced expressions are not in general unique. It is also known that there exists a unique element $w_0 \in W$ for which $\ell$ is maximal; it is called the {\it longest element of} $W$. This element is also the maximum of $W$, considered as a poset with the {\it Bruhat order}, which can be defined as follows: two elements $\tau,w\in W$ satisfy that $\tau\leq w$ if and only if some substring of a reduced expression of $w$ is a reduced expression of $\tau$.

\begin{notation}\label{not:words}
If $w=s_{i_1}s_{i_2} \dots s_{i_m}$ is a reduced expression of $w\in W$, for $0 \le k <m$ we will define 
\begin{itemize}
\item $w\lef{k}:=s_{i_1}s_{i_2} \dots s_{i_k}$, which has length $k$;
\item $w\rig{k}:=s_{i_{k+1}} \dots s_{i_m}$, with length $m-k$;
\end{itemize}
so that we may write:
\[w=w\lef{k} w\rig{k},\quad\mbox{for every }w\in W.\]
Note that this decomposition depends on the choice of a reduced expression of $w$.
\end{notation}

Let us also recall the following invariant of the group $W$:

\begin{definition}\label{def:coxnum}
Let $W$ be the Weyl group of a semisimple algebraic group $G$, with Dynkin diagram $\cD$. Then the {\it Coxeter number of} $W$, denoted by $h(\cD)$, is the order of any element of the form:
$$
s_{\sigma(1)}\dots s_{\sigma(n)}, \mbox{ with $\sigma$ a permutation of $\{1,\dots,n\}$.}
$$
The Coxeter number of $W$ can be defined in several equivalent ways, that show its importance as an invariant of finite Coxeter groups and Lie algebras (see \cite[3.16--3.20]{Hum3} and the references therein).
\end{definition}

\begin{definition}\label{def:WJ}
Given a subset $J\subset \Delta$, let $W_J$ be the following subgroup of $W$:
\[W_J:=\langle s_i \ | \ i\in J \rangle \subset W\] 
This is a Weyl group, whose longest element we denote by $w_{0J}$.
Following \cite[Section 5.12]{Hum3} we define
\[W^J:=\{w \in W \ | \ \ell(ws_i) > \ell(w)\ \text{for all} \ i \in J\}\] 
\end{definition}

Any element of $W$ can be decomposed uniquely as a product of an element of $W_J$ and an element of $W^J$:

\begin{lemma}\cite[Section 5.12]{Hum3}\label{lem:dec}
Given $v \in W$ there exist unique $v^J \in W^J$ and $v_J \in W_J$ such that $v =v^Jv_J$ and  $\ell(v)=\ell(v^J) + \ell(v_J)$. Moreover, $v^J$ is the unique element of smallest length in the left coset $vW_J$.
\end{lemma}


In general the subset $W^J\subset W$ is not a subgroup. By \cite[Corollary 2.5.3]{BjBr}, it has a unique element of maximal length, 
that we denote by $w_0^J$. Furthermore, we may observe the following:

\begin{remark}\label{rem:declong} In the case of the longest element $w_0$, the decomposition given in Lemma \ref{lem:dec} has the following properties (see \cite[formula (2.16)]{BjBr}): \[ (w_0)^J = w_0^J \qquad (w_0)_J =w_{0J}. \]
\end{remark}

\subsection{Schubert varieties}\label{ssec:Schubert}

Let us recall that for any algebraic variety $M$ we have a map between the Chow ring and the cohomology ring:
\begin{equation}\label{eq:chowandhom}
\AAA^\bullet(M)\lra \HH^\bullet(M,\Z),	
\end{equation}
which is a homomorphism of graded rings which doubles the degree (cf. \cite[Corollary 19.2]{Fult}). If $M$ is a rational homogeneous variety, the Bruhat decomposition of the corresponding semisimple group defines a cellular decomposition of $M$, and  the above map is an isomorphism  (see \cite[Example 19.1.11]{Fult}). Furthermore, the classes of the closures of the cells -- called {\it Schubert varieties} -- constitute a basis of the free abelian group $\AAA^\bullet(M)$. We will describe now Schubert varieties, considering separately the case $M=G/B$ and the case $M=G/P_J$, $J\neq\emptyset$.

\subsubsection{Schubert varieties in complete flag manifolds}\label{sssec:SchubertGB}

\begin{definition}\label{def:Schub}
Given $w \in W$ we define the {\it Schubert variety} $X_w$ as the closure of the $B$-orbit $BwB/B$ associated to $w$. For instance, $X_{w_0}=G/B$. It is known that the dimension of $X_w$ is equal to $\ell(w)$, and its codimension in $G/B$ is then:
\[c(w):=\codim(X_w,G/B)=\ell(w_0)-\ell(w).\] 
The classes of Schubert varieties in $\AAA^\bullet(G/B)$ are called {\it Schubert cycles} in $G/B$.
\end{definition}

\begin{remark}\label{rem:Schub}
In order to identify every Schubert variety as a union of cells, we note that the cellular decomposition is obtained by projecting to $G/B$ the Bruhat decomposition of $G$ (\cite[Section~28.3]{Hum1}),
\[G=\bigsqcup_{w\in W}BwB.\]
Since we have that $\overline{BwB}=\bigsqcup_{\tau \leq w}B\tau B$, then we get that $X_w=\bigsqcup_{\tau \leq w}B\tau B/B$.
\end{remark}

\begin{definition}\label{def:oppositeSchubert}
Denoting by $B^-:=w_0Bw_0$ the opposite Borel subgroup of $B$, the closure of 
$B^-wB/B\subset G/B$ is called the {\it opposite Schubert variety} $X^w$. \end{definition}

\begin{remark}\label{rem:oppoSchub}
The relation between Schubert varieties and opposite Schubert varieties is the following. For every $w\in W$:  
\begin{equation}\label{eq:oppoSchub}
X^w=w_0X_{w_0w}.
\end{equation}
The variety $X^w$ has dimension $c(w)$ and codimension $\ell(w)$. 
In particular, we have the following equality of classes in $\AAA^\bullet(G/B)$:
\begin{equation}\label{eq:oppoSchub2}
[X^v]=[X_{w_0v}]\in \AAA^{\ell(v)}(G/B) \quad \text{ and } \quad [X_v]=[X^{w_0v}]\in \AAA^{c(v)}(G/B).\end{equation} 
\end{remark}

Let us recall some known facts that will allow us to write $\ed(G/B)$ in terms of products of Schubert cycles and intersections of Schubert varieties:
\begin{proposition}\label{prop:generalfacts}
Consider the complete flag manifold $G/B$:
\begin{enumerate}
	\item[(1)] The Schubert cycles form a $\Z$-basis of $\AAA^{\bullet}(G/B)$.
	\item[(2)] The cones of effective classes in $\AAA^c(G/B)\otimes_\Z\Q$ are polyhedral cones generated by the Schubert cycles of codimension $c$, for every $c\geq 0$. 
	\item[(3)] The product of two Schubert cycles can be written as an integral combination of Schubert cycles: 
	\[
	[X_u]\cdot [X_v]= \sum_{\substack{w\in W\\
	c(u)+c(v)=c(w)}} c_{uv}^w[X_w]
	\] 
	with nonnegative coefficients $c_{uv}^w$.
	\item[(4)] Given $u,v\in W$, the varieties $X_u$, $X^v$ meet properly -- i.e., every component of their intersection has the expected codimension $\ell(v)+c(u)$--  if and only if  $v \le u$. In particular, if $\ell(v)+c(u) \le \dim G/B$ then
	\begin{equation*}\label{eq:fund}
	 [X_u]\cdot [X^v]\neq 0 \Longleftrightarrow 
	 v \le u.
	 \end{equation*}
\end{enumerate}  
\end{proposition}
\begin{proof} Applying \cite[Theorem 1]{FMSS} to the unipotent part of $B$ we get (1) and (2). 
For (3) and (4) see  \cite[Propositions 1.3.6 and  1.3.2]{Brion2005}.
\end{proof}

We can thus re-write the effective good divisibility of the complete flag manifold $G/B$ as a property of the group $W$:
\begin{lemma}\label{lem:toprove}
The variety $G/B$ has effective good divisibility up to degree $s$ if and only if for every two elements $u,v\in W$ such that $\ell(v)+c(u) = s$ we have $v \le u$.
\end{lemma}

\begin{proof}
	The statement follows straightforwardly from Proposition \ref{prop:generalfacts}. Let us write $x_i$ and $x_j$ as positive integral linear combinations of Schubert cycles using (1) and (2). Then, by (3) $x_i x_j$ is a sum of nonnegative multiples of Schubert cycles as is shown; we conclude by (4).
\end{proof}

\subsubsection{Schubert varieties in rational homogeneous varieties}\label{sssec:SchubertGP}

The Bruhat decomposition also defines a cellular decomposition of any rational homogeneous variety of the form $G/P_J$, with $J \subset \Delta$, as follows:

\begin{definition}\label{def:schubertGP}
For any $w \in W^J$ we define the {\it Schubert variety} $X_{w}\subset G/P_J$ as the closure of $C_{w} := BwP_J/P_J \subset G/P_J$, and the {\it opposite Schubert variety} $X^{w}\subset G/P_J$ as the closure of $C^{w} := B^-wP_J/P_J \subset G/P_J$; their dimensions are, respectively, $\ell(w)$ and $\dim(G/P_J)-\ell(w)$.  
As in the case of $G/B$, the classes of Schubert varieties in $\AAA^\bullet(G/P_J)$  are called {\it Schubert cycles}.
\end{definition}

\begin{remark}\label{rem:schubertGP}
	In particular we have $X^w= w_0X_{(w_0w)^J}$, for every $w\in W^J$. Therefore we may write: 
	\begin{equation}\label{eq:updown}
		[X^w] = [X_{(w_0w)^J}], \quad \text{and} \quad [X_w] = [X^{(w_0w)^J}].
	\end{equation}
\end{remark}

\begin{remark}\label{rem:chowrings}
We have used the same notation for Schubert varieties in $G/B$ and $G/P_J$. Let us discuss the relation among them, by means of the natural projection $\pi:G/B\to G/P_J$. Note first that, an element in $G/B$ can be written in the form $bwB$ for a unique $w\in W$, and some $b\in B$. By Lemma \ref{lem:dec}, $w$ can be decomposed uniquely in the form $w^Jw_J$, with $w^J\in W^J$, $w_J\in W_J$; then \cite[Corollary~28.3]{Hum1} implies that the image of $bwB$ into $G/P_J$ is equal to $bw^JP_J$, and, in particular: 
\[
\pi(X_w)=X_{w^J}.
\]
Moreover, the pullback map: 
\[
\pi^*:\AAA^\bullet(G/P_J)\lra \AAA^\bullet (G/B),
\]
 is a monomorphism of rings; in fact, since these two rings can be identified with the corresponding cohomology rings, see (\ref{eq:chowandhom}), the injectivity of the map follows from  \cite[Corollary 5.4]{BGG}. More concretely, for every $w\in W^J$, the inverse image of $X_{w}\subset G/P_J$ is
\begin{equation}\label{eq:chowrings}
\pi^{-1}(X_{w})=\bigsqcup_{\substack{\tau\in W^J,\,\,
\tau\leq w\\\sigma\leq w_{0J}}}B\tau\sigma B =X_{ww_{0J}}.
\end{equation}
Note that $\ell(ww_{0J})=\ell(w)+\ell(w_{0J})$ by Lemma \ref{lem:dec}.  

Moreover, using that $w_0ww_{0J} \in W^J$ for every $w \in W^J$ (see \cite[Proposition 2.5.4]{BjBr}), we may write the following equalities of subvarieties of $G/P_J$:
\begin{equation}\label{eq:oS}
X^v=w_0X_{(w_0v)^J}=w_0X_{(w_0vw_{0J})^J} =w_0X_{w_0vw_{0J}},
\end{equation}
hence, for $v\in W^J$, we get that the pullback of $X^v\subset G/P_J$ is $X^v\subset G/B$:
\begin{equation}\label{eq:chowrings2} 
\pi^{-1}(X^v) =w_0 \pi^{-1}(X_{w_0vw_{0J}})=w_0(X_{w_0vw_{0J}w_{0J}})=w_0X_{w_0v} =X^v.\end{equation}

In particular, $\codim(X^v,G/P_J)=\codim(X^v,G/B)=\ell(v)$.

\end{remark}

Since, given $I\subset J\subset \Delta$ the monomorphism $\pi^*:\AAA^\bullet(G/P_J)\lra \AAA^\bullet (G/P_I)$, induced by the natural projection $G/P_I\to G/P_J$, preserves effectivity 
we may conclude the following: 

\begin{corollary}\label{cor:pull} If  $I\subset J\subset \Delta$, then $\ed(G/P_I) \le \ed(G/P_J)$. In particular $\ed(G/B)
\le  \ed(G/P_J)$ for every $J \subset \Delta$.
\end{corollary}

Finally, as in the case of $G/B$, the properties of intersection of Schubert and opposite Schubert varieties provide enough information  to understand the products of effective cycles, and can be translated in terms of the Bruhat ordering. 

\begin{notation}\label{not:codim}
For any $u \in W$, we define
\[c^J(u) := \ell(w_0^J)-\ell(u^J), \qquad c_J(u) := \ell(w_{0J})-\ell(u_J). 
\]
\end{notation}

The integer $c^J(u)=c^J(u^J)$ is the codimension of the Schubert variety $X_{u^J} \subset G/P_J$, while $c_J(u)=c_J(u_J)$ is the codimension of the Schubert variety $X_{u_J}$ in $P_J/B$.
As a consequence of Lemma \ref{lem:dec} and Remark \ref{rem:declong}
we have
\begin{equation}\label{eq:cod}
c(u) = c^J(u) +c_J(u).
\end{equation}

The following statement is an extension of Proposition \ref{prop:generalfacts} to any rational homogeneous variety:

\begin{proposition}\label{prop:generalfactsGP}
	Consider the rational homogeneous variety $G/P_J$:
	\begin{enumerate}
		\item[(1)] The Schubert cycles form a $\Z$-basis of $\AAA^{\bullet}(G/P_J)$.
		\item[(2)] The cones of effective classes in $\AAA^c(G/P_J)\otimes_\Z\Q$ 
 are polyhedral cones generated by the Schubert cycles of codimension $c$, for every $c\geq 0$.
		\item[(3)] The product of two Schubert cycles $X_u, X_v$, $u,v\in W^J$, 
		can be written as an integral combination of Schubert cycles: 
	\[
	[X_u]\cdot [X_v]= \sum_{\substack{w\in W^J\\
	c^J(u)+c^J(v)=c^J(w)}} c_{uv}^w[X_w]
	\] 
	with nonnegative coefficients $c_{uv}^w$.
		\item[(4)] Given $u,v\in W^J$, the varieties $X_u$, $X^v$ meet properly -- i.e., every component of their intersection has the expected codimension $\ell(v)+c^J(u)$--  if and only if  $v \le u$. In particular, if $\ell(v)+c^J(u) \le \dim G/P_J$ then
	\begin{equation*}\label{eq:fundJ}
	 [X_u]\cdot [X^v]\neq 0 \Longleftrightarrow 
	 v \le u.
\end{equation*}
	\end{enumerate}  
\end{proposition}

\begin{proof} As in the case of complete flags, (1) follows from \cite{BGG}, and (2), (3) from \cite{FMSS}.
For the last part we will use part (4) of Proposition \ref{prop:generalfacts}. 

Let $\pi:G/B\to G/P_J$ be the natural projection. Note first that, by Remark \ref{rem:chowrings}, $\pi^{-1}(X_u)=X_{uw_{0J}}$. On the other hand, by formula (\ref{eq:oS}),
$$
X^v=w_0 X_{w_0vw_{0J}}=w_0\pi(X_{w_0vw_{0J}}).
$$
Then, by Remark \ref{rem:chowrings}, 
$X_u\cap X^v\neq\emptyset$ if and only if $X_{uw_{0J}}\cap w_0X_{w_0vw_{0J}}=X_{uw_{0J}}\cap w_0X^{vw_{0J}}\neq\emptyset$ in $G/B$ , which (by Proposition \ref{prop:generalfacts}) is equivalent to $vw_{0J}\leq uw_{0J}$. We conclude the proof by noting that this is equivalent to  $v\leq u$. 

Note first that, if $w \in W^J$ then,  adding to a reduced expression of $w$ a reduced expression of $w_{0J}$, we obtain, by Lemma \ref{lem:dec}, a reduced expression of $ww_{0J}$. 

Assume first that $v\leq vw_{0J}\leq uw_{0J}$.  
The fact that $v\leq uw_{0J}$ implies that the reduced expression of $uw_{0J}$ contains a reduced expression of $v$. But the last element of a reduced expression of $v$ cannot be of the form $s_j$, $j\in J$ by the definition of $W^J$; then it is part of the reduced expression of $u$. This implies that $v\leq u$. Conversely, if $v\leq u$, then clearly $vw_{0J} \le uw_{0J}$.
\end{proof}	

While the coefficients $c_{uv}^w$ appearing in Proposition  \ref{prop:generalfactsGP} (3)  do not depend only on $W^J$, 
 Proposition \ref{prop:generalfactsGP} (4)  implies that the nonvanishing of the product $ [X_u]\cdot [X^v]$ is determined by the Bruhat order on $W^J$; in particular $\ed(G/P_J)$ depends only on the poset $W^J$: 

\begin{lemma}\label{lem:toprove2}
The variety $G/P_J$ has effective good divisibility up to degree   $s$ if and only if for every two elements $u,v\in W^J$ such that $\ell(v)+c^J(u) = s$ we have $v \le u$.
\end{lemma}

In particular, note that, with the standard numbering of the nodes of their Dynkin diagrams, the map sending each reflection $s_i$ of $\DB_n$ to the reflection $s_i$ of $\DC_n$ provides a group isomorphism between the corresponding Weyl groups preserving the Bruhat order.  Hence we get the following:

\begin{corollary}\label{cor:BnCn}
For every $I\subset \Delta$, we have:
$$
\ed(\DB_n(I))=\ed(\DC_n(I)).
$$
\end{corollary}

\begin{remark}\label{rem:Kleiman} 
The equality in Proposition \ref{prop:generalfactsGP} (3) is known as the {\it Littlewood--Richardson rule}, since it is a generalization of the classical formula of the decomposition of the product of two Schur functions. Given an md-pair in a rational homogeneous variety, the nonnegativity of the coefficients in that rule (which essentially follows from item (2) in the Proposition)  allows us to find an md-pair consisting of two Schubert cycles. Then, by \cite[Theorem~1.7]{3264}, we may claim that these two Schubert cycles are represented by two disjoint subvarieties. 
\end{remark}

\section{Effective good divisibility of  flag manifolds}

This Section is devoted to the proof of  Theorem \ref{thm:1} which, in view of Corollary \ref{cor:BnCn}, is reduced to  
the cases $\cD=\DA_n, \DB_n$ or $\DD_n$. 
In Section \ref{ssec:upper} we will show that the effective good divisibility of a complete flag manifold of classical type $\cD$ is smaller than or equal to $\cox(\cD)-1$. In fact, we will provide examples of rational homogeneous varieties of Picard number one for which the effective good divisibility is smaller than or equal to $\cox(\cD)-1$; the bound will then follow from Corollary \ref{cor:pull}. In Section \ref{ssec:lower} we will conclude the proof of the Theorem by showing that the provided bound is the exact value of the effective good divisibility in each case.

\subsection{Upper bounds} 
\label{ssec:upper}

The fact that $\ed(\DA_n(r))=n$ ($=\cox(\DA_n)-1$) for every $r$ has been proven in \cite{NO22}. 
In the following examples we will use some projective geometry arguments to compute upper bounds in the cases of rational homogeneous varieties of Picard number one of types $\DB_n$ and $\DD_n$. 

\begin{example}\label{ex:bnr} \mbox{\boldmath{$\DB_n(r)$}}: Grassmannian of $(r-1)$-dimensional linear spaces contained in the $(2n-1)$-dimensional quadric $\DB_n(1)$.
We  consider the subvariety $\Sigma_p\subset \DB_n(r)$, parametrizing  $(r-1)$-dimensional linear spaces contained in $\DB_n(1)$ and passing through a point $p$ in the quadric $\DB_n(1)$, and the closed subset $\Sigma_H\subset \DB_n(r)$, parametrizing  
linear spaces contained in a smooth hyperplane section $H$ of $\DB_n(1)$ not containing $p$. Clearly $[\Sigma_p]\cdot[\Sigma_H]=0$. Note that $\Sigma_p$ is a Schubert variety in $\DB_n(r)$, while $\Sigma_H$ is not (however $[\Sigma_H]$ is a Schubert cycle).

The variety $\Sigma_p$ is isomorphic to $\DB_{n-1}(r-1)$, and  $\Sigma_H$ parametrizes $(r-1)$-dimensional linear spaces in a $(2n-2)$-dimensional smooth quadric (in particular it is nonempty, but not necessarily irreducible). 
Then the  sum of the codimensions of $\Sigma_p$ and $\Sigma_H$ is
\begin{equation*} 2\dim \DB_n(r) -(\dim \Sigma_p + \dim \Sigma_H)=
r(4n + 1 - 3r) - \dfrac{(r-1)(4n-3r)}{2} -\dfrac{r(4n - 1 - 3r)}{2} = 2n.
\end{equation*}
We have thus shown that $\ed(\DB_n(r)) \le 2n-1$ for every $r,n$. \qed
\end{example}

\begin{example}\label{ex:dnr} \mbox{\boldmath{$\DD_n(r),\; r\in \{2,\dots ,n-2\}$}}: Grassmannian of $(r-1)$-dimensional linear spaces contained in the $(2n-2)$-dimensional quadric $\DD_n(1)$.  
The same argument as in Example \ref{ex:bnr}  provides here the inequality $\ed(\DD_n(r)) \le 2n-2$ for every $r\leq n-2$, $n\geq 4$. \end{example}

\begin{example}\label{ex:dnsp} \mbox{\boldmath{$\DD_n(r),\; r\in \{1,n-1 ,n\}$}}: 
	the $(2n-2)$-dimensional quadric $\DD_n(1)$ or a Grassmannian of 
	 $(n-1)$-dimensional linear spaces contained in $\DD_n(1)$, classically known as a spinor variety.  

We start with the case $r=1$ observing that the $(2n-2)$-dimensional quadric $\DD_n(1)$ contains two disjoint linear subspaces of dimension $n-1$, therefore we get $\ed(\DD_n(1)) \le 2n-3$. Note that these two linear subspaces belong to the same class in $\AAA^{\bullet}(\DD_n(1))$ if $n$ is even, and to different classes if $n$ is odd (cf. Remark \ref{rem:updown}). 

For the remaining cases we note first that $\DD_n(n-1)\simeq\DD_n(n)$, so we may assume $r=n$. Given two general points in $\DD_n(1)$, there is no positive dimensional linear space contained in $\DD_n(1)$ passing through them. Hence, if we denote by $\Sigma_p$ the subvariety of $\DD_n(n)$ parametrizing  $(n-1)$-dimensional linear spaces contained in $\DD_n(1)$ and passing through a point $p$, we have $[\Sigma_p]\cdot[\Sigma_p]=0$.
The variety $\Sigma_p$ is isomorphic to $\DD_{n-1}(n-1)$, so it has codimension
\[ \dim \DD_n(n) -\dim \DD_{n-1}(n-1) = 
\dfrac{n(n-1)}{2} -\dfrac{(n-1)(n-2)}{2} = n-1
\]
This shows that $\ed(\DD_n(n-1))=\ed(\DD_n(n)) \le 2n-3$. \qed

\end{example}
We summarize in Table \ref{tab:b1} the bounds obtained in Examples \ref{ex:bnr}, \ref{ex:dnr} and \ref{ex:dnsp}. 
 
\begin{table}[h!]
\caption{Bounds on $\ed(\cD(r))$.}
\centering
\begin{tabular}{|C||C|C|}
\hline
\cD(r) & \ed(\cD(r))&\\
\hline
\DA_n(r) &  = n&\\
\hline
\DB_n(r) /\DC_n(r)&  \le 2n-1&\\
\hline
\DD_n(r)  &\le 2n-3 & r=1,n-1,n\\
\hline
\DD_n(r)  &\le 2n-2 & r\not =1,n-1,n\\
\hline
\end{tabular}
\label{tab:b1}
\end{table}

Together with Corollary \ref{cor:pull} and Table \ref{tab:cox}, this provides the following:

\begin{corollary}\label{cor:edle} Let $G/B$ be a complete flag manifold of classical type $\cD$. Then $\ed(G/B) \le \cox(\cD) -1$.
\end{corollary}

\subsection{Proof of Theorem \ref{thm:1}}\label{ssec:lower}

To show that $\ed(G/B) \ge \cox(\cD) -1$ we will use an inductive argument. We set $J:= \Delta \setminus \{1\}$, and denote by $\cD_J$ the subdiagram of $\cD$ supported on the nodes indexed by $J$. A fundamental ingredient in the proof will be a description of some reduced expressions for the elements of maximal length in $W^J$ provided by Stumbo in \cite{Stumbo}, that we recall here. 

\begin{example}\label{ex:projquad}
We consider the case of rational homogeneous varieties $\cD(1)$, where $\cD=\DA_n,\DB_n,\DC_n,\DD_n$ is a Dynkin diagram of classical type, and set $J:= \Delta \setminus \{1\}$.  In \cite[Theorems 2,4 and 6]{Stumbo} (note that in \cite{Stumbo} the nodes are numbered differently) it has been shown that 
the maximal elements $w_0^J \in W^J$ have the reduced expressions listed in Table \ref{tab:redexp}.

\begin{table}[h!!]
\caption{Reduced expressions of $w_0^J$.\label{tab:redexp}}
\centering
\begin{tabular}{|C||C|}
\hline
\cD & w_0^J\\
\hline
\DA_n &  s_ns_{n-1}\dots s_{2}s_1\\
\hline
\DB_n / \DC_n&  s_1s_2\dots s_{n-1}s_ns_{n-1} \dots s_2s_1\\
\hline
\DD_n  &s_1s_2\dots s_{n-2}s_ns_{n-1} \dots s_2s_1=s_1s_2\dots s_{n-2}s_{n-1}s_{n} \dots s_2s_1 \\
\hline
\end{tabular}
\end{table}

Moreover, Stumbo shows that every element in $W^J$ has a reduced expression which is a {\it right} substring of the above reduced expressions of $w_0^J$. Note that, in the $\DD_n$ case, the fact that $s_{n-1}s_n=s_ns_{n-1}$, implies that, for $m\neq n-1$:
$$
(s_1s_2\dots s_{n-2}s_ns_{n-1} \dots s_2s_1)\rig{m}=(s_1s_2\dots s_{n-2}s_{n-1}s_{n} \dots s_2s_1)\rig{m}. 
$$
Then we may describe some generators of $\AAA^\bullet(\cD(1))$ as follows:
\begin{align}
\AAA^{m}(\cD(1))&=\Z [X_{w_0^J\rig{m}}] \hspace{2.845cm} \mbox{for }\cD=\DA_n,\DB_n,\DC_n ,\  \quad \nonumber\\
\AAA^{m}(\DD_n(1))&= \begin{cases}\Z [X_{w_0^J\rig{m}}] &\qquad\qquad  \mbox{for }m \neq n-1, \\[4pt]
\Z ([X_{w_\alpha}],[X_{w_\beta}]) &\qquad\qquad  \mbox{for }m = n-1,\end{cases}
 \nonumber
\end{align}
where
\[w_\alpha := s_{n-1}s_{n-2}\dots s_2s_1, \qquad \qquad  w_\beta:= s_{n}s_{n-2}\dots s_2s_1.
\]
\end{example}

The following statement is a straightforward consequence of Example \ref{ex:projquad}:

\begin{lemma}\label{lem:strem}  The following statements hold:
\begin{enumerate}
\item[(0)] $\ell(w_0^J)$ is equal to $\cox(\cD)-1$ if $\cD \not = \DD_n$, and to $\cox(\cD)$ if $\cD=\DD_n$;
\item[(1)] if $v\in W$ has length one, 
 then $v \le w_0^J$; 
\item[(2)] in cases $\DB_n/\DD_n$, if $v\in W$ has length two, 
then
\begin{enumerate} 
\item $v \le w_0^J$, and
\item  $v \le w_0^J\rig{1}$, unless $v=s_1s_2$;
\end{enumerate}
\item[(3)]
 if $v^J, u^J \in W^J$ with $\ell(v^J) \le \ell(u^J)$, then $v^J \le u^J$ except in case $\DD_n$ when $\{v^J,u^J\}= \{w_\alpha, w_\beta\}$. 
\end{enumerate}
\end{lemma}

\begin{proof}
Using the results quoted in Example \ref{ex:projquad}, the proof goes through a case-by-case analysis of the right substrings of the elements $w_0^J$. We just  note that, in cases $\DB_n/\DD_n$, $w_0^J\rig{1}$ obviously contains every possible reduced expression of length two of the form $s_is_j$ with $i>1$; since $s_1s_i=s_is_1$ for $i>2$, then the result follows. 
\end{proof}

\begin{corollary}\label{lem:J}  If $v,u \in W$ satisfy that $\ell(v) +c(u) =  \cox(\cD)-1$,  
then $v^J \le u^J$.\end{corollary}

\begin{proof}
Using formula (\ref{eq:cod}), the assumption  $\ell(v) +c(u) =  \cox(\cD)-1$ and item (0) in Lemma \ref{lem:strem}, 
we can write:
\begin{align}\nonumber 
 \ell(v^J)  \le\ell(v) &=  \ell(w_0^J) -c(u) \le \ell(w_0^J)-c^J(u) =      \ell(u^J) & \cD & \not = \DD_n \\ \nonumber
 \ell(v^J)  \le\ell(v) &=  \ell(w_0^J) -c(u)-1 \le \ell(w_0^J)-c^J(u) -1=      \ell(u^J)-1 & \cD & = \DD_n  
\end{align}
so the result follows from Lemma \ref{lem:strem} (3).
\end{proof}

\begin{theorem}\label{main1}
A complete flag manifold of classical type $\cD$ has  effective good divisibility up to degree $\cox(\cD)-1$.
\end{theorem}

\begin{proof} As pointed out in Lemma \ref{lem:toprove}, we need to prove that  $v \le u$ for every two elements $u,v\in W$ such that  $\ell(v)+c(u)= h(\cD)-1$. 

We will prove the statement by induction on the number of nodes of $\cD$, starting from the cases $\DA_1, \DB_2$ and $\DD_4$. For $\DA_1$ the statement is clear, and for $\DB_2$ and $\DD_4$ it can be easily checked by listing all the possible words (see Appendix \ref{ssec:div}). Set 
\[m:=\cox(\cD)-\cox(\cD_J)=\begin{cases} 1 & \text{if }\cD=\DA_n,\\
2  &  \text{if }\cD\not=\DA_n.\end{cases}
\]

\medskip
\noindent{\bf Case 1.}  Assume $v=v_J$.\par
\medskip
If $\ell(v_J) +c_J(u) \le \cox(\cD_{J})-1$, then we  use induction to get $v =v_J \le u_J \le u$. \\ 
If $\ell(v_J) +c_J(u) > \cox(\cD_{J})-1$, then we have
\[\cox(\cD)-1= \ell(v_J)+c_J(u)+c^J(u) > \cox(\cD_{J})-1 +c^J(u),
\]
and in particular $c^J(u) \le m-1 \le 1$. Since $u^J=w_0^J\rig{c^J(u)}$, and $s_1 \not \le v$,
by Lemma \ref{lem:strem} (1) and (2), we get $v\lef{m} \le u^J$.

On the other hand, by induction, since 
\[\ell(v\rig{m}) +c_J(u)  \le \ell(v) -m+ c(u)  
=\cox(\cD)-1-m  =\cox(\cD_J)-1,
\]
we get $v\rig{m} \le u_J$, and we can  conclude that $v = v\lef{m}v\rig{m}  \le u^Ju_J =u$.\par
\medskip
\noindent{\bf Case 2.} Assume $v\not=v_J$; in particular $\ell(v_J) \le \ell(v) -1$.\par 
\medskip
If $\ell(v_J) + c_J(u) \le \cox(\cD_{J})-1$, then we conclude that $v_J \le u_J$ by induction. By Corollary \ref{lem:J} in all these cases we have also $v^J \le u^J$, therefore $v \le u$.

If the above inequality is not satisfied, then $m=2$, $\ell(v_J) = \ell(v) -1$ and $c_J(u) =c(u)$, that is,
$\cD \not = \DA_n$, $v^J=s_1$, and $u^J=w_0^J$. 
Since
 \[\ell (v_J\rig{1}) +c_J(u) \le \ell(v)-2 +c(u)=\cox(\cD_J)-1,\]
 we can apply induction to get 
$v_J\rig{1} \le u_J$. 
 
On the other hand, by Lemma \ref{lem:strem} (2) $v^J v_J\lef{1}\le u^J =w_0^J$. We can thus conclude that $v = v^Jv_J\lef{1}v_J\rig{1}  \le u^Ju_J =u$.
\end{proof}

\section{Proof of Theorem \ref{thm:2}}\label{sec:Dn}

In this section we will compute the effective good divisibility of rational homogeneous varieties of classical type and Picard number one. We start by observing that for varieties of type $\DA_n,\DB_n,\DC_n$ and for some varieties of type $\DD_n$ this follows from Theorem \ref{thm:1}.

\begin{corollary}\label{cor:corthm1}
Let $\cD(R)$ be a rational homogeneous variety, with either $\cD=\DA_n,\DB_n,\DC_n$, $R\subset\{1,\dots,n\}$, or $\cD=\DD_n$ and $R\cap\{1,n-1,n\}\neq \emptyset$. Then
$$
\ed(\cD(R))=\cox(\cD)-1.
$$ 
\end{corollary}

\begin{proof}
As a consequence of the examples presented in Section \ref{ssec:upper} we get that $\ed(\cD(r))\le\cox(\cD)-1$ (see Table \ref{tab:b1}), for some $r\in R$. On the other hand, by Corollary \ref{cor:pull}, and Theorem \ref{thm:1},  in every case
$$
\cox(\cD)-1=\ed(G/B)\leq  \ed (\cD(R))\leq \ed(\cD(r))\le\cox(\cD)-1.
$$
This concludes the proof.
\end{proof}

The rest of the section will be devoted to the computation of $\ed(\DD_n(R))$, when $R\cap \{1,n-1,n\}=\emptyset$. The proof will be completed in Section \ref{ssec:proofthm2} below. A fundamental ingredient of the proof will be the fact that the only \edps of Schubert cycles  in $\DD_n(1),\DD_n(n-1)$, and $\DD_n(n)$ are those described in Example \ref{ex:dnsp}, namely the classes of two disjoint linear subspaces of dimension $n-1$ in $\DD_n(1)$, and the class of the subset of  $\DD_n(r)$ ($r=n-1,n$) parametrizing linear subspaces passing by a point of $\DD_n(1)$ (which has zero self-intersection). We start by proving this in the following section.

\subsection{Maximal disjoint pairs in quadrics and spinor varieties} \label{ssec:dn1}

Let $\Delta$ be the set of nodes $\Delta$ of the diagram $\cD=\DD_n$. Along this section we will use the following notation: 
\[J:=\Delta \setminus \{1\},\quad I:=\Delta \setminus \{n\},\quad K:=I \cap J.\]
We also recall from Example \ref{ex:projquad} the elements 
\[
w_\alpha := s_{n-1}s_{n-2}\dots s_2s_1, \qquad w_\beta:= s_{n}s_{n-2}\dots s_2s_1.
\]
of the Weyl group $W$ of $\DD_n$.

\subsubsection{Maximal disjoint pairs in quadrics}\label{sssec:dn1}

Let us consider the smooth $(2n-2)$-dimensional quadric $\cD(1)$. This variety contains two families of linear subspaces of dimension $n-1$, parametrized by the spinor varieties $\cD(n-1)$ and $\cD(n)$. 
Let us  describe the corresponding classes of these linear spaces in $\DA^{n-1}(\cD(1))$:

\begin{lemma}\label{lem:linspquad}
Let $\Lambda_\alpha\in \cD(n)$, $\Lambda_\beta\in \cD(n-1)$ be linear spaces of dimension $n-1$ contained in $\cD(1)$. 
Then
\[
[\Lambda_\alpha] = [X_{w_\alpha}] ,\qquad [\Lambda_\beta] = [X_{w_\beta}] .
\]
\end{lemma}

\begin{proof} Let us prove the statement for $\Lambda_\alpha$; the argument for $\Lambda_\beta$ is analogous. This linear space can be seen as the image into $\cD(1)$ of the fiber of the natural map $G/B \to \cD(n)$ over one point. This fiber belongs to the class $[X_v]$ in $\AAA^\bullet(G/B)$ where $v:=w_{0I}$ is the longest element of $W_I$. Note that $W_I$ is the Weyl group of the subdiagram $\cD_I=\DA_{n-1}\subset \cD$, obtained by deleting the $n$-th node. We then write $v=v^Kv_K$ with $v^K\in (W_I)^K$, $v_K\in (W_I)_K=W_K$ (Lemma \ref{lem:dec}) and,  by Remark \ref{rem:chowrings}, $[\Lambda_\alpha]=[X_{v^K}]\in \AAA^\bullet(\cD(1))$.   

  By Remark \ref{rem:declong}, $v_K = w_{0K}$, and $v^K$ is the maximal element of the poset $(W_I)^K$. Since the set $K$ is the complement of the first node in $\cD_I$, by Example \ref{ex:projquad} we get  $v^K= s_{n-1}s_{n-2}\dots s_1$, which is precisely $w_\alpha$. \end{proof}

\begin{lemma}\label{lem:longdn} 
Let $w_0$ be the longest element in the Weyl group $W$ of $\cD$; then
\begin{equation}\label{eq:omegaalpha}
	\begin{cases}
		w_0w_\alpha  = w_\alpha w_0 {\hbox { if }} n {\hbox { is even }}\\
		w_0w_\alpha  = w_\beta w_0 {\hbox { if }} n {\hbox { is odd }}\\
	\end{cases}
\end{equation}
\end{lemma}
\begin{proof}
By item (IX) in \cite[Planche~IV]{Bourb}  the conjugation with $w_0$ in $W$ is the identity if $n$ is even, while it fixes $s_1, \dots, s_{n-2}$ and interchanges $s_{n-1}$ and $s_n$ if $n$ is odd. Then the result follows by definition of $w_\alpha$ and $w_\beta$.
\end{proof}

\begin{remark}\label{rem:updown}
By Equation (\ref{eq:updown}) and Example \ref{ex:projquad},  we have:
\begin{equation}\label{eq:ab}
[X_{w_\alpha}] = [X^{(w_0w_\alpha)^J}]= \begin{cases} [X^{w_\alpha w_0^J}] =  [X^{w_\beta}]&\text{if} \ n \ \text{is even} \\
[ X^{w_\beta w_0^J}] =[X^{w_\alpha}] & \text{if} \ n \ \text{is odd} \end{cases}
\end{equation}
In particular we re-obtain in the language of reduced expressions the well known properties of the intersection of linear subspaces of maximal dimension contained in a smooth $(2n-2)$-dimensional quadric:
\[ [X_{w_\beta}] \cdot  [X_{w_\alpha}]  =
 \begin{cases}  [X_{w_\beta }]\cdot [X^{w_\beta}] \not = 0 &  \text{if } n \ \text{is even}\\
 [X_{w_\beta }]\cdot [X^{w_\alpha}]  = 0 &  \text{if }n \ \text{is odd}\end{cases}
\]
\[[X_{w_\alpha}] \cdot  [X_{w_\alpha}] = 
 \begin{cases}  [X_{w_\alpha }]\cdot [X^{w_\beta}] = 0 &  \text{if } n \ \text{is even}\\
 [X_{w_\alpha}]\cdot [X^{w_\alpha}]  \not = 0 &  \text{if }n \ \text{is odd}\end{cases}
\]
\[[X_{w_\beta}] \cdot  [X_{w_\beta}] = 
 \begin{cases}  [X_{w_\beta }]\cdot [X^{w_\alpha}] = 0 &  \text{if } n \ \text{is even}\\
 [X_{w_\beta}]\cdot [X^{w_\beta}]  \not = 0 &  \text{if }n \ \text{is odd}\end{cases}
\]
\end{remark}

Furthermore we may now describe in our language all the \edps of Schubert cycles in quadrics of even dimension:

\begin{proposition}\label{prop:uniqueinquad}
If $u^J,v^J\in W^J$, then $([X_{u^J}],[X^{v^J}])$ is an \edp in $\cD(1)$ if and only if $\{u^J,v^J\} =\{w_\alpha, w_\beta\}$. 
\end{proposition}

\begin{proof} Let $u^J,v^J$ be two elements in $W^J$, and assume, without loss of generality, that $\ell(v^J)\leq \ell(u^J)$.  Recalling Proposition \ref{prop:generalfactsGP} (4), we have to show that $v^J \le u^J$ unless $\{u^J,v^J\} =\{w_\alpha, w_\beta\}$ and this follows from Lemma \ref{lem:strem} (3).
\end{proof}

Pulling back an \edp of  Schubert cycles in $\cD(1)$, we get an \edp of Schubert cycles in $G/B$. By Proposition \ref{prop:uniqueinquad} these \edps can be characterized as follows:

\begin{lemma}\label{lem:d1}
An \edp$([X_u],[X^v])$ in $G/B=\cD(\Delta)$ is a pullback from $\cD(1)$ if and only if $\{v^J,u^J\}= \{w_\alpha, w_\beta\}$. In particular $\ell(v)=c(u)=n-1$.
\end{lemma}

\begin{proof} 
Note first that if $([X_u],[X^v])$ is an \edp in $G/B$, then $c(u)+\ell(v)=2n-2$ by Theorem \ref{thm:1}. If $\{v^J,u^J\}= \{w_\alpha, w_\beta\}$, then 
$c^J(u)=n-1$, $\ell(v^J)=n-1$.  Then it follows that $c_J(u)=0$ and $\ell(v_J)=0$, i.e. $u_J=w_{0J}$ and $v=v^J$. In particular, by Remark \ref{rem:chowrings}, we get that $X_u$ is a pullback from $\cD(1)$. 
 We  show now that $X^v=w_0X_{w_0v}=w_0X_{w_0v^J}$ is a pullback, as well. In the case $v=w_\alpha$ -- the case of $v=w_\beta$ is analogous -- this follows from (\ref{eq:omegaalpha}), which implies that:
\[ w_0w_\alpha = \begin{cases} w_\alpha w_0 =   w_\alpha w_0^Jw_{0J}=w_\beta w_{0J}& \text{if } n \ \text{is even} \\
w_\beta w_0 = w_\beta w_0^Jw_{0J}=   w_\alpha w_{0J} & \text{if }n \ \text{is odd} \end{cases}
\]
To prove the other implication, up to replacing $u$ and $v$ with $w_0u$, $w_0v$, we can assume that $\ell(v) \le c(u)$. Then,
if the pair $([X_u],[X^v])$ is a pullback we have $v_J  \le u_J = w_{0J}$, so  $[X_u] \cdot [X^v]=0$ implies 
$v^J \not \le u^J$. Since
\[
 \ell(v^J)  \le\ell(v) =  \ell(w_0^J) -c(u) \le \ell(w_0^J)-c^J(u)=      \ell(u^J),
\] 
by Lemma \ref{lem:strem} (3) we must have  $\{v^J,u^J\}= \{w_\alpha, w_\beta\}$.
\end{proof}

\subsubsection{Maximal disjoint pairs in spinor varieties} \label{ssec:dnn}

We will consider now the spinor variety $\cD(n)$ (the case of $\cD(n-1)$ is analogous);  we define
\[\theta_\beta:= s_1 \dots s_{n-3}s_{n-2}s_n=w_\beta^{-1}, \qquad \theta_{\alpha} := s_1 \dots s_{n-3}s_{n-2}s_{n-1}=w_\alpha^{-1},\]
 and, for every $l\le n-1$:
\[\theta_{\alpha}(l):=\theta_{\alpha}\rig{n-l-1}=(\omega_\alpha\lef{l})^{-1},\qquad \theta_{\beta}(l):=\theta_{\beta}\rig{n-l-1}=(\omega_\beta\lef{l})^{-1},\]
which are (reduced) 
right subexpressions of length $l$ of $\theta_{\alpha},\theta_{\beta}$.
Then, as shown in \cite[page 707]{Stumbo}, setting $\gamma=\beta$ if $n$ is even and $\gamma=\alpha$ if $n$ is odd,  every element of $W^I$ has a reduced expression of the form 
\[\theta_\gamma(l_1) \dots \theta_{\beta}(l_{n-3}) \theta_{\alpha}(l_{n-2})\theta_{\beta}(l_{n-1}),\]
where $(l_1, l_2, \dots, l_{n-1})$ is either the sequence $(1,2,\dots,n-1)$, corresponding to $w_0^I$, or a non decreasing $(n-1)$-tuple of nonnegative integers of the form
\begin{equation}\label{eq:WI}
0=l_1=\dots =l_k<l_{k+1}<\dots<l_{n-1}\le n-1.
\end{equation}

Of particular interest  is the element corresponding to the sequence $(0,1, \dots, n-3,n-2)$, that we denote by  $\sigma_\beta$; the class $[X_{\sigma_\beta}]$ is the class of a subvariety $\Sigma_p$ (cf. Example \ref{ex:dnsp}), parametrizing linear spaces of maximal dimension in $\DD_n(1)$ passing  through a point $p$.  In fact, it is the image in $\cD(n)$ of the fiber of $G/B \to \cD(1)$ over  $p$; this fiber belongs to the class $[X_v]$ in $\AAA^\bullet(G/B)$ where $v=w_{0J}$ is the longest element of $W_J$.  Note that $W_J$ is the Weyl group of the subdiagram $\cD_J=\DD_{n-1}\subset \cD$, obtained by deleting the first node of $\cD$.
We write $v=v^Kv_K$ with $v^K\in (W_J)^K$, $v_K\in (W_J)_K=W_K$ (Lemma \ref{lem:dec}). 
By Remark \ref{rem:declong}, $v_K = w_{0K}$;  and $v^K$ is the maximal element of the poset $(W_J)^K$, so it is the element of $(W_J)^K$ corresponding to the sequence $(1, \dots, n-3,n-2)$, hence $v^{K}= \sigma_\beta$.

\begin{lemma} In the Weyl group of $\cD$ we have that
\[
\sigma_\beta= \begin{cases} \theta_\beta w_0^I &  n \ \text{even}\\
\theta_\alpha w_0^I &  n \ \text{odd}
\end{cases}
\]
\end{lemma}

\begin{proof}
We will prove that $w_0^I=\theta_\beta^{-1}\sigma_\beta$ if $n$ is even. The proof that $w_0^I=\theta_\alpha^{-1}\sigma_\beta$ if $n$ is odd is similar. Write
\[\theta_\beta^{-1}\sigma_\beta=s_ns_{n-2} \dots s_2 s_1 \theta_\alpha(1) \theta_\beta(2) \dots \theta_{\alpha}(n-3)\theta_{\beta}(n-2) \]
Since $s_i$ commutes with all the reflections in $\theta_\gamma(l)$ if $l \le n-1-i$ we can write
\[\theta_\beta^{-1}\sigma_\beta=s_ns_{n-2} \theta_\alpha(1) s_{n-3}\theta_\beta(2) \dots s_2\theta_{\alpha}(n-3)s_1\theta_{\beta}(n-2) \]
Observing that $s_i\theta_\gamma(n-i-1) = \theta_\gamma(n-i)$ for $i \le n-2$ and that $s_n=\theta_\beta(1)$ we get 
\[\theta_\beta^{-1}\sigma_\beta= \theta_\beta(1)  \theta_\alpha(2) \theta_\beta(3) \dots \theta_{\alpha}(n-2)\theta_{\beta}(n-1) =w_0^I \qedhere \]
\end{proof}

By Lemma \ref{lem:longdn}, for $n$ even we have $ w_0 \theta_\beta=\theta_\beta w_0$ while, for $n$ odd we have $w_0\theta_\beta  = \theta_\alpha w_0$. Therefore, in all cases 
\begin{equation}\label{eq:thetasigma}(w_0\theta_\beta)^I=  (\theta_\gamma w_0^Iw_{0I})^I = \theta_\gamma w_0^I = \sigma_\beta,\end{equation}
hence, by formula (\ref{eq:updown}), $
 [X^{\theta_\beta}]=[X_{\sigma_\beta}]$,
from which we deduce
\[[X_{\sigma_\beta}]\cdot[X_{\sigma_\beta }]  = [X_{\sigma_\beta}]\cdot[X^{\theta_\beta}] =0.\]
The geometric explanation of the above formula is the following: there are no lines, hence no linear spaces, passing through two general points of $\cD(1)$ and contained in it.

The next statement describes some properties of \edps in  $\cD(\Delta)$ that are not pullbacks from $\cD(1)$. We need to exclude here the case $n=4$, for which the statement does not hold, but whose \edps can be computed directly using {\tt SageMath} (cf. Appendix \ref{ssec:ed}). 

\begin{proposition}\label{prop:d2} Let $([X_u], [X^v])$ be an \edp in $\cD(\Delta)$ which is not a pullback from $\cD(1)$, and $n\geq 5$. If  $\ell(v) \le c(u)$ then  $\ell(v)=c(u)$, $c^J(u)=\ell(v^J)=1$, and $([X_{u_J}], [X^{v_J}])$ is an \edp for  $\cD_J(J)$.
\end{proposition}

\begin{proof} First of all we notice that, by Lemma \ref{lem:d1} and Lemma \ref{lem:strem} (3), we have $v^J \leq u^J$.
We claim that $c^J(u) + \ell(v^J) \le 2$. If this were not the case, then
\[c_J(u)+\ell(v_J) \leq 2(n-1)-3=\cox(\cD_J)-1,\] 
so that $v_J \le u_J$ by Theorem \ref{thm:1} and Proposition \ref{prop:generalfacts}(4). It would then follow that $v \le u$, against the assumption $[X_u] \cdot [X^v] =0$.\par

We will prove the statement by induction, starting from the case $n=5$, in which the Proposition can be proved by listing all the possible cases (see Appendix \ref{ssec:ed}). 
\par
\medskip
\noindent {\bf Step 1. } Prove that $c^J(u)\neq 0$.\par
\medskip
Assume the contrary. 
Set $m= 2- \ell(v^J)$.  Note that 
\begin{equation}\ell(v_J\rig{m}) = \ell(v) -2 <  c(u)=c_J(u)\label{eq:forindu}\end{equation}
 and 
\begin{equation}\ell(v_J\rig{m}) + c_J(u)=\ell(v)-2+c(u)= 2(n-2). \label{eq:forindu2}\end{equation}

If $[X_{u_J}] \cdot [X^{v_J\rig{m}}]= 0$, then $([X_{u_J}],[X^{v_J\rig{m}}])$ is an \edp for the diagram $\cD_J(J)$. It is not a pullback from the quadric $\cD_J(2)$ otherwise Lemma \ref{lem:d1} would tell us that 
$$n-2=\ell(v_J\rig{m}) =\ell(v_J)-m=  \ell(v) -2,$$ contradicting the assumption $\ell(v) \le n-1$. 
Then we may apply induction on $n$, to claim that $\ell(v_J\rig{m})=c(u_J)=c_J(u)$, which contradicts (\ref{eq:forindu}); we conclude that $[X_{u_J}] \cdot [X^{v_J\rig{m}}]\neq 0$ and so, by Proposition \ref{prop:generalfacts} (4) $v_J\rig{m} \le u_J$.

On the other hand, $v^Jv_J\lef{m}$ has length two hence,  by Lemma \ref{lem:strem} (2.a), we have $v^Jv_J\lef{m} \le u^J = w_0^J$, and so $v=v^Jv_J\lef{m} v_J\rig{m}\le u$, contradicting the fact that $([X_u], [X^v])$ is an \edp.\par
\medskip
\noindent {\bf Step 2.}  Prove that $\ell(v^J) \neq 0$. \par
\medskip
Assume the contrary. In particular  $s_1 \not \le v$. By Step 1, we know that $c^J(u)=1$ or $2$; set $m=3-c^J(u)$ which is equal to $2$ or $1$, respectively. Since $u^J=w_0^J\rig{c^J(u)}$, by the description of $w_0^J$ in Example \ref{ex:projquad} and Lemma \ref{lem:strem} (2.b) we get $v\lef{m} \le u^J$.
On the other hand, we have
 \[\ell(v\rig{m}) +c_J(u) = \ell(v)-m+c_J(u) = \ell(v) + c(u)-3=\cox(\cD_J)-1\]
so, by Theorem \ref{thm:1} and Proposition \ref{prop:generalfacts}(4), we get $v\rig{m} \le u_J$, and conclude that $ v\leq u$, contradicting the fact that $([X_u], [X^v])$ is an md-pair.\par
\medskip
We have thus shown that $c^J(u),\ell(v^J)\neq 0$, so -- by the claim at the beginning of the proof -- necessarily $(c^J(u),\ell(v^J))=(1,1)$. This implies that $v^J=s_1$,  $u^J =w_0^J\rig{1}$, $\ell(v_J)=\ell(v)-1$, and $c_J(u)=c(u)-1$.

By direct inspection of the reduced expressions of the word $u^J =w_0^J\rig{1}$, we have $v^J \le u^J$. If  $[X_{u_J}] \cdot [X^{v_J}]\not =0$ then 
 $v_J  \le u_J$ and $v \le u$, a contradiction. 
 
Hence $[X_{u_J}] \cdot [X^{v_J}]=0$, i.e., $([X_{u_J}], [X^{v_J}])$ is an \edp for $\cD_J(J)$ and, by Lemma \ref{lem:d1} and induction, we get
$\ell(v_J)=c_J(u) = n-2$, so $\ell(v)=c(u)=n-1$. 
\end{proof}

\begin{corollary}\label{cor:disjoint}  If $([X_u],[X^v])$ is an \edp for $\cD(n)$ then $(u,v)=(\sigma_\beta,\theta_\beta)$, or $n=4$ and $(u,v)=(\sigma_\beta,\theta_\beta)$ or $(s_1s_2s_4,\theta_\alpha)$.
\end{corollary}


\begin{proof}
For $n=4$, the description of the \edps in $\DD_4(4)$ can be obtained from Proposition \ref{prop:uniqueinquad} by permuting the indices $1$ and $4$ in the diagram $\DD_4$.

For $n\geq 5$, we set $u'=uw_{0I}$; the pair $([X_{u'}],[X^v])$ is an \edp for $\cD(\Delta)$ which is not a pullback from $\cD(1)$, hence, by Proposition \ref{prop:d2}, we have $\ell(v)=c(u)=n-1$ and $v^J=s_1$. From the description of elements of $W^I$ given in \ref{ssec:dnn} we see that the only element of length $n-1$ containing $s_1$ is $\theta_\beta$, hence $v=\theta_\beta$ is the element corresponding to the sequence $a=(0, \dots, 0, n-1)$. Let  $b=(b_1, \dots, b_{n-1})$ be the sequence corresponding to $u$. We have 
 \[\sum b_i=\ell(u)=\dim \cD(n)-n-1= \frac{(n-1)(n-2)}{2}\]
and $b_{n-1} < a_{n-1}=n-1$ since $([X_{u}],[X^v])$ is an md-pair. This forces  $b=(0,1, \dots, n-3,n-2)$, that is, $u=\sigma_\beta$. 
\end{proof}

%
%
%
%

\subsection{Effective good divisibility of $\DD_n$-varieties}\label{ssec:proofthm2}

In this section we will conclude the proof of Theorem \ref{thm:2}. The main ingredient here will be  the following:

\begin{proposition}\label{prop:dnfail}
The only  \edps $([X_u],[X^v])$ for $\cD(\Delta)$  are pullbacks from $\cD(1),\cD(n-1)$ or $\cD(n)$.
\end{proposition}

In fact, with this statement at hand we may now argue as follows:

\begin{proof}[Proof of Theorem \ref{thm:2}]
From Example \ref{ex:dnr}, Theorem \ref{thm:1} and Corollary \ref{cor:pull}, we know that, for any $R \subset \Delta$, 
\[2n-3=\ed(\cD(\Delta))\leq\ed(\cD(R))\leq 2n-2,\] 
If $\ed(\cD(R))= 2n-3$, pulling back to $\cD(\Delta)$ an \edp  for $\cD(R)$ we get an \edp for $\cD(\Delta)$, and, by Proposition \ref{prop:dnfail},  $R$ meets $\{1,n-1,n\}$.
\end{proof}

Before proving Proposition \ref{prop:dnfail}, let us show the following:

\begin{lemma}\label{lem:pullpull}
Let $([X_u],[X^v])$ be an \edp for $\cD(\Delta)$; if  $X^v$ is a pullback from $\cD(r)$, $r=1,n-1$ or $n$, then also $X_u$ is a pullback from the same variety.
\end{lemma}

\begin{proof}
 Let $\pi:\cD(\Delta)\to \cD(r)$ be the natural projection. If $X^v$ is a pullback from $\cD(r)$, then $X_u \cap X^v =\emptyset$ if and only if $\pi(X_u)\cap \pi(X^v) =\emptyset$. 

Since $X^v$ is a pullback  then $\codim(\pi(X^v),\cD(r))=\codim(X^v,\cD(\Delta))$, while $\codim(\pi(X_u),\cD(r))\geq\codim(X_u,\cD(\Delta))$ and equality holds if and only if $X_u$ is a pullback. Since the effective good divisibility of $\cD(\Delta)$ and $\cD(r)$ are equal, the conclusion follows.
\end{proof}

\begin{proof}[Proof of Proposition \ref{prop:dnfail}]
We start by noting that in the cases $n=4,5$ the statement can be proved by listing all the possible \edps and checking that they are pullbacks from $\cD(1),\cD(n-1)$, or $\cD(n)$. (see Appendix, \ref{ssec:ed} and \ref{ssec:pull}). We prove the statement for $n\geq 6$ by induction.

By Lemma \ref{lem:d1} and Proposition \ref{prop:d2} either $([X_{u}],[X^{v}])$ is a pullback from $\cD(1)$ or  
 $\ell(v)=c(u)=n-1$, $(c^J(u),\ell(v^J))=(1,1)$ (in particular $v^J=s_1$), and $([X_{u_J}],[X^{v_J}])$ is an \edp for $\cD_J$. By induction $([X_{u_J}],[X^{v_J}])$ is a pullback from $\cD_J(2),\cD_J(n-1)$ or $\cD_J(n)$.\par
\medskip
If $X_{u_J},X^{v_J}$ were pullbacks from $\cD_J(2)$, by Lemma \ref{lem:d1} we would have, setting $J':=J \setminus \{2\}$, $w'_\alpha:=s_{n-1}s_{n-2} \dots s_2$ and $w'_\beta:=s_{n}s_{n-2} \dots s_2$, that
\[(v_J,u_J)= (w'_\alpha, w'_\alpha w_{0J'})\ \text{or} \ (w'_\beta, w'_\beta w_{0J'}).\]
Therefore, since $v^J = s_1$ we would get
\[
v= s_{n}s_{n-2} \dots s_3 s_1 s_2, \quad \text{or} \quad v= s_{n-1}s_{n-2} \dots s_3 s_1 s_2,
\]
hence $v\lef{n-2} \le u^J =w_0^J\rig{1}$ (see Example \ref{ex:projquad}). 
 In every case $v\rig{n-2}=s_2 \le u_J$, therefore $v \le u$, against the assumptions.\par
 \medskip
If $X_{u_J},X^{v_J}$ are pullbacks from $\cD_J(n)$, then $v_J = \theta'_\beta= s_2 \dots s_{n-3}s_{n-2}s_{n}$ and $u_J=\sigma'_\beta w_{0K}$ (see Corollary \ref{cor:disjoint}).
It follows that $v=s_1v_J=\theta_\beta$ and $w_0v= \sigma_\beta w_{0I}$ (see Equation (\ref{eq:thetasigma})). In particular  $[X^v]=[X_{w_0v}]=[X_{\sigma_\beta w_{0I}}]$   is a pullback from $\cD(n)$. By Lemma \ref{lem:pullpull} $X_u$ is a pullback from $\cD(n)$, too.\par\medskip

If  $X_{u_J},X^{v_J}$  are pullbacks from $\cD_J(n-1)$ the argument is analogous.
\end{proof}

\section{Morphisms to rational homogeneous varieties}\label{sec:applications}

In this section we are going to prove Theorem \ref{teo:maps2b}. This result is a refined, more general version of \cite[Proposition~4.7]{MOS6}, in which the proof was done by using Chern classes of pullbacks of universal bundles.

\begin{remark}\label{rem:maps2b}
For every $r\in R$, denote by $\pi_r$ the natural map $\cD(R)\to \cD(r)$. Since $\cD(R)$ can be identified with a subvariety of $\prod_{r\in R}\cD(r)$ so that the maps $\pi_r$ are restrictions of the natural projections, then $f:M\to \cD(R)$ is constant if and only if $\pi_r\circ f$ is constant for every $r$. As a consequence, in the setup of Theorem \ref{teo:maps2b} we may assume, without loss of generality, that $R$ consists of a single element $r$.
\end{remark}

The key result to relate $\ed(M)$ with the existence of nonconstant morphisms from $M$ to a homogeneous variety (in particular to a rational homogeneous varieties $G/P$) is the following:

\begin{lemma}\label{lem:nosurj} Let $M$ be a smooth complex projective variety, $X$ a complex variety with a transitive action of a connected algebraic group $G$, and  $\varphi:M \to 
X$ a morphism, with image $M'$. If $[A]\in \AAA^a(X), [B] \in \AAA^b(X)$ are effective classes satisfying $[A]\cdot [B]=0$ and $a+b \le \ed(M)$ then, either $[M'] \cdot [A] = 0$, or 
$[M'] \cdot [B] = 0$; in particular $\varphi$ is not surjective.
\end{lemma}

\begin{proof}
For a general $g,g' \in G$, $gA$ and $g'B$  are disjoint and generically transverse to $\varphi$, that is, $\varphi^{-1}(gA), \varphi^{-1}(g'B)$ are generically reduced and of the same codimensions as $A,B\subset X$ (see \cite[Theorem 1.7]{3264}). Note that the existence of the transitive action of $G$ on $X$ implies that $X$ is quasi-projective and, in particular, we may apply \cite[Theorem 1.23]{3264}, to get  that  $\varphi^*[A]=[\varphi^{-1}(gA)]$ and $\varphi^*[B]=[\varphi^{-1}(g'B)]$; since $\varphi^*[A]\cdot \varphi^*[B]=0$, by the assumptions on $\ed(M)$ either $[\varphi^{-1}(gA)]$ or $[\varphi^{-1}(g'B)]$ is the zero class in $M$. We conclude by noting that $[\varphi^{-1}(gA)]=0$ is equivalent to  $[M'] \cdot [A] =0$.
\end{proof}

\begin{corollary}\label{cor:nosurj0} 
Let $M$ and $M'$ be complex projective varieties, with $M$ smooth and $\ed(M)> \dim M'$. Then there are no nonconstant morphisms $\varphi:M \to M'$.
\end{corollary}

\begin{proof}
Let $f:M \to M'$ be a morphism; replacing $M'$ with the image of $f$ we may assume that $f$ is surjective. Composing with a finite morphism $g:M' \to \PP^{\dim M'}$ we get a morphism $\varphi:M \to \PP^{\dim M'}$, which is constant by Lemma \ref{lem:nosurj}. It follows that $f$ is constant.
\end{proof}

\begin{corollary}\label{cor:nosurj1} If $\cD$ is a diagram of classical type and $M$ is a smooth complex projective variety such that $\ed(M)> \ed(\cD(1))$ then there are no nonconstant morphisms $\varphi:M \to \cD(1)$. 
\end{corollary}

\begin{proof}
 
In view of Lemma \ref{lem:nosurj} a morphism $\varphi:M \twoheadrightarrow M' \subseteq \cD(1)$ cannot be surjective. Since, in all cases $\ed(\cD(1)) \ge \dim \cD(1)-1$ (see Theorem \ref{thm:2}), we get $\ed(\cD(1))\geq  \dim M'$, hence 
 $\ed(M)> \dim M'$; by Corollary \ref{cor:nosurj0}  we get that 
$\varphi$ is constant.
\end{proof}

Let us now describe, in the following example, 
two rational maps that will be used in the Proof of Theorem \ref{teo:maps2b}.

\begin{example}\label{ex:2raz}
Let $p \in \PP^n$ be a point, $H$ a hyperplane not containing $p$, and $r$ an integer $2 \le r \le n$. We will consider two rational maps defined on $\DA_n(r)$.

To unify the notation we will use the convention that $\DA_{n-1}(n)$ is a  single point.
\begin{enumerate}
\item The linear projection from $p$ onto $H$ induces a rational map $$\pi_p:\DA_n(r) \dashrightarrow \DA_{n-1}(r),$$ 
sending $\Lambda$ to $\pi_p(\Lambda)$.  It is not defined in the subvariety $\Sigma_p \subset \DA_n(r)$, parametrizing $(r-1)$-dimensional linear spaces passing by $p$. 
The  fiber of $\pi_p$ 
 over $\Lambda \in  \DA_{n-1}(r)$ consists of the $(r-1)$-dimensional linear spaces contained in $\langle \Lambda, p \rangle$ and not containing $p$, therefore  is an affine space $\mathbb A^{r}$.

\item We can also define a rational map $$\pi_H:\DA_n(r) \dashrightarrow \DA_{n-1}(r-1),$$ sending $\Lambda$ to $\Lambda \cap H$. This map is not defined on the subvariety $\Sigma_H \subset \DA_n(r)$, parametrizing $(r-1)$-dimensional linear spaces contained in $H$. The fiber of $\pi_H$ 
 over $\Lambda \in  \DA_{n-1}(r-1)$ consists of the $(r-1)$-dimensional linear spaces containing $\Lambda$ and not contained in $H$,  therefore is an affine space $\mathbb A^{n-r+1}$.

\end{enumerate}
\end{example}

The fact that the fibers of the above maps are affine provides the following:

\begin{lemma}\label{lem:2raz}
Let $r,n$ be integers such that  $2 \le r \le n$, let $\varphi:M \twoheadrightarrow M' \subseteq \DA_n(r)$ a morphism from a projective variety $M$. Let $\Sigma_p,\Sigma_H\subseteq \DA_n(r)$ be as in Example \ref{ex:2raz}. 
\begin{enumerate}
\item  If $M' \cap \Sigma_p= \emptyset$ and $\pi_p \circ \varphi$ is constant, then $\varphi$ is constant.
\item If $M' \cap \Sigma_H= \emptyset$ and $\pi_H \circ \varphi$ is constant, then $\varphi$ is constant.
\end{enumerate}
\end{lemma}

We are now ready to prove the main result of the section:

\begin{proof}[Proof of Theorem \ref{teo:maps2b}] By Remark \ref{rem:maps2b} we can consider only morphisms $M \to \cD(r)$ to varieties of Picard number one and, by Corollary \ref{cor:nosurj1}, we can assume that $r \ge 2$.\par
\medskip
\noindent {\bf Case 1. } $\cD$ of type $\DA$.
\par\medskip
Let $n$ be the minimum integer such that there exists a nonconstant morphism $\varphi:M \twoheadrightarrow M' \subseteq \DA_n(r)$ for some $r$.
We will show that $n=\ed(\DA_n(r))\geq \ed(M)$. Assume by contradiction that $n < \ed(M)$.


Let $p \in \PP^n$ be a general point, $H \subset \PP^n$ a general hyperplane, and define $\Sigma_p$, $\Sigma_H$ as in Example \ref{ex:2raz}. Then clearly $[\Sigma_p] \cdot [\Sigma_H]=0$. The variety $\Sigma_p$ is isomorphic to
$\DA_{n-1}(r-1)$, so it has codimension $n+1-r$ in $\DA_n(r)$ while
$\Sigma_H$ is isomorphic to $\DA_{n-1}(r)$, so it has codimension $r$ in $\DA_n(r)$.
Therefore $\codim \Sigma_p + \codim \Sigma_H = n+1 \le \ed(M)$. 
 By Lemma \ref{lem:nosurj} we then have that either $[M']\cdot [\Sigma_p]=0$ or $[M']\cdot [\Sigma_H]=0$.

In the first case,  $\pi_p \circ \varphi$ is a morphism from $M$ to $A_{n-1}(r)$,
which is constant by our choice of $n$;  we can then apply Lemma \ref{lem:2raz} to get that $\varphi$ is constant.
In the second case we consider  $\pi_H \circ \varphi$ and use a similar argument.\par
\medskip
\noindent {\bf Case 2. } $\cD=\DB_n, \DC_n$ or $\DD_n$.
\par\medskip

Let $\cD(1) \subseteq \PP^N$ be  the minimal embedding of $\cD(1)$ and $p \in \cD(1)$ a general point, $H \subset \PP^N$ a general hyperplane and $\Sigma_p$, $\Sigma_H$ the subvarieties of $\cD(r)$ parametrizing $(r-1)$-dimensional linear spaces in $\cD(1)$ passing through $p$ or, respectively, contained in $H$.  Let $\varphi:M \twoheadrightarrow M' \subseteq \cD(r)$ be a nonconstant morphism. 

Clearly $[\Sigma_p] \cdot [\Sigma_H]=0$ and in all cases but $\cD=\DD_n, r=n-1,n$ we have (cf. Examples \ref{ex:bnr} and \ref{ex:dnr})
\[\codim \Sigma_p + \codim \Sigma_H =\ed(\cD(r)) +1 \le \ed(M).\]

By Lemma \ref{lem:nosurj} the intersection of  $[M']$ with $[\Sigma_p]$ or $[\Sigma_H]$ is zero.
 In the cases $\cD=\DD_n, r=n-1,n$ we have $[\Sigma_p] \cdot [\Sigma_p]=0$ and $\codim \Sigma_p + \codim \Sigma_p =\ed(\cD(r)) +1 \le \ed(M)$, hence, by Lemma \ref{lem:nosurj} we have $[M']\cdot [\Sigma_p]=0$. 

If $[M']\cdot [\Sigma_p]=0$ we consider the restriction to $M' \subset \cD(r) \subset \DA_N(r)$ 
of the morphism $\pi_p$
defined in Example \ref{ex:2raz}; composing with $\varphi$ we get a morphism $M \to \DA_{N-1}(r)$, which is constant by Case 1. Then $\varphi$ is constant by Lemma \ref{lem:2raz}.

If $[M']\cdot [\Sigma_H]=0$ we consider  $\pi_H \circ \varphi : M \to \DA_{N-1}(r-1)$, and we conclude by a similar argument.
\end{proof}

\section*{Acknowledgments}
The authors would like to thank two anonymous referees for their valuable comments, which helped improving the exposition of this paper.

\subsection*{Conflict of interest}

The authors declare no potential conflict of interests.

\appendix

\section{Computations and codes}

As we have seen in the previous sections, our arguments require some computations of effective good divisibility and \edps for low rank groups. These computations can be done by using some built-in commands of the {\tt SageMath} software (cf. \cite{sagemath}). 
We have included in this appendix the scripts that we have used for our purposes.

\subsection{Divisibility in a fixed degree}\label{ssec:div}

We choose a Weyl group with Dynkin diagram $\cD$ and a degree $d$, and we check if a complete flag manifold of type $\cD$ has $\ed$ up to degree $d$. 
For instance, in the case $\DD_4$ and $d= \cox(\DD_4)-1=5$ to check that $\ed(\DD_4(\Delta)) \ge 5$, we set

{\footnotesize
\begin{lstlisting}
W = WeylGroup("D4", prefix="r"); d=5
\end{lstlisting}
}

We define a function which computes the list of elements of a given length.
{\footnotesize
\begin{lstlisting}
def lists(u) : 
    l1=[]
    for w in W.elements_of_length(u):
        l1.append(w.reduced_word())
    return l1 
\end{lstlisting}
}

We define the longest element, its length (which is the dimension of the flag manifold $G/B$)
and we compute two lists of words: the list of words $v$ of length $\ell(v) \le e:=\lfloor d/2 \rfloor$,
and the list of words $u$ of length $\ell_0-c(u)$ such that $e \le c(u) <d$.

{\footnotesize
\begin{lstlisting}
w0=W.long_element(); l0=len(w0.reduced_word())
e =floor(d/2)

words=[0]; upwords=[0]
for i in [1..e]:
    words.append([])
for i in [1..(d-1)]:
    upwords.append([])     

for i in [1..e]:
    words[i]=lists(i)
for i in [e..(d-1)]:
    upwords[i]=lists(l0-i)
\end{lstlisting}
}

We verify the condition  $v \le u$ for every pair of words such that $\ell(v)+c(u) = d$ such that $0<\ell(v) \le c(u)$.

\begin{lstlisting}
bad=0
for i in [1..e]:     
    for w1 in words[i]:   
        for w2 in upwords[d-i]:
            v=W.from_reduced_word(w1) 
            u=W.from_reduced_word(w2)
            if v.bruhat_le(u) == False:
                bad =bad +1
if bad == 0:
    print W.cartan_type(), "has e.d. greater or equal than", d
if bad > 0:
    print W.cartan_type(), "has e.d. smaller than", d 
\end{lstlisting}

For instance, in case $\DD_4$, checking the degrees $d=5$ and $d=6$ we find that $\ed(\DD_4(\Delta))= 5$.

\subsection{Md-pairs}\label{ssec:ed}

The following code finds the pairs of words such that $\ell(v)+c(u) = d$ such that $0<\ell(v) \le c(u)$ and  $v \not \le u$; the list is empty if $d \le \ed(\cD(\Delta))$ and gives the whole list of \edps if $d = \ed(\cD(\Delta))+1$.

{\footnotesize
\begin{lstlisting}
c=1
for i in [1..e]:
    for w1 in words[i]:
        for w2 in upwords[d-i]:
            v=W.from_reduced_word(w1) 
            u=W.from_reduced_word(w2)
            if v.bruhat_le(u) == False:
                print c,")l(v)=",i,"c(u)=", d-i," v=", w1, "u=", w2 
                c= c+1
\end{lstlisting}
}

For $\DD_4$ in degree $d=6$ we find:

{\footnotesize
\begin{lstlisting}
1)   l(v)= 3  c(u)= 3  v=[1, 2, 3]  u=[4, 2, 3, 1, 2, 4, 1, 2, 1]
2)   l(v)= 3  c(u)= 3  v=[1, 2, 4]  u=[3, 2, 4, 1, 2, 3, 1, 2, 1]
3)   l(v)= 3  c(u)= 3  v=[3, 2, 1]  u=[4, 2, 3, 1, 2, 4, 2, 3, 2]
4)   l(v)= 3  c(u)= 3  v=[3, 2, 4]  u=[1, 2, 4, 1, 2, 3, 1, 2, 1]
5)   l(v)= 3  c(u)= 3  v=[4, 2, 1]  u=[2, 3, 1, 2, 4, 1, 2, 3, 2]
6)   l(v)= 3  c(u)= 3  v=[4, 2, 3]  u=[2, 3, 1, 2, 4, 3, 1, 2, 1]
\end{lstlisting}
}

For $\DD_5$ in degree $d=8$ we find:

\begin{lstlisting}
1) l(v)=4 c(u)=4 v=[1,2,3,4] u= [4,3,5,2,3,4,1,2,3,5,1,2,3,1,2,1]
2) l(v)=4 c(u)=4 v=[1,2,3,5] u= [5,3,4,2,3,5,1,2,3,4,1,2,3,1,2,1]
3) l(v)=4 c(u)=4 v=[4,3,2,1] u= [3,5,2,3,4,1,2,3,5,1,2,3,4,2,3,2]
4) l(v)=4 c(u)=4 v=[5,3,2,1] u= [4,3,5,2,3,4,1,2,3,5,2,3,4,2,3,2]
\end{lstlisting}

\subsection{Pullbacks}\label{ssec:pull}

We define the function {\tt dec}, which computes the decomposition of an element $w \in W$ as $w^Jw_J$, where $J \subset \Delta$. This can be used to check if a Schubert cycle $X_w \subset \cD(\Delta)$ is the pullback of a Schubert cycle  $X_{w^J}\subset \cD(J)$ (see Remark \ref{rem:declong}). 
 
\begin{lstlisting}
def desc(w,J) :
    for i in J:
        s=[i]
        u = (W.from_reduced_word(w))*(W.from_reduced_word(s))
        if len(u.reduced_word())<len(w):
            return u.reduced_word(),i    
    return w,0  
    
def dec(w,J) :
    p=[0]; pj=[0]; 
    p.append(desc(w,J)[0])
    pj.append(desc(w,J)[1])
    for i in [1..l0+1]: 
        if pj[i]==0:
            return  p[i], pj[1:i]
        p.append(desc(p[i],J)[0])
        pj.append(desc(p[i],J)[1])
\end{lstlisting} 

Using this function we can check that, for the six \edps listed for $\DD_4$,  cases  3) and 5)  are pullbacks from $\DD_4(1)$, cases 1) and 6) are pullbacks from  $\DD_4(3)$, and cases 2) and 4) are pullbacks from $\DD_4(4)$.

We can also check that, in case $\DD_5$, the pairs 3) and 4) are pullbacks from $\DD_5(1)$, while cases 1) and 2) satisfy the properties listed in Proposition \ref{prop:d2}. For instance, applying {\tt dec} for  $J=\{2,3,4,5\}$ to the reduced expression $u=[4, 3, 5, 2, 3, 4, 1, 2,$ $ 3, 5, 1, 2, 3, 1, 2, 1]$ we get
$u^J=[2, 3, 5, 4, 3, 2, 1]$ and  $u_J=[2, 3, 2, 5, 3, 2, 4, 3, 5]$.
A further application of {\tt dec} for  $K=\{3,4,5\}$ to $u_J$ shows that $u_J$ is a pullback from the first node of the diagram $\DD_4$ obtained removing the node $1$.

\end{document}